\documentclass{article}

\usepackage{fullpage,color,todonotes}
\usepackage{amsmath,amsfonts,amssymb,amsthm,thmtools,thm-restate,enumitem,hyperref,cleveref}

\newtheorem{theorem}{Theorem}[section]
\newtheorem{lemma}[theorem]{Lemma}

\theoremstyle{definition}

\providecommand{\id}{\mathsf{id}}
\renewcommand{\neg}{\mathsf{neg}}
\providecommand{\tNAE}{\text{3NAE}}
\DeclareMathOperator{\dom}{dom}

\title{Aggregation of evaluations without unanimity}
\author{Yuval Filmus\footnote{Technion Israel Institute of Technology. Research supported by ISF grant 507/24.}}

\begin{document}

\maketitle

\begin{abstract}
Dokow and Holzman determined which predicates over $\{0, 1\}$ satisfy an analog of Arrow's theorem: all unanimous aggregators are dictatorial. Szegedy and Xu, extending earlier work of Dokow and Holzman, extended this to predicates over arbitrary finite alphabets.

Mossel extended Arrow's theorem in an orthogonal direction, determining all aggregators without the assumption of unanimity. We bring together both threads of research by extending the results of Dokow--Holzman and Szegedy--Xu to the setting of Mossel.
As an application, we determine, for each symmetric predicate over $\{0,1\}$, all of its aggregators.
\end{abstract}

\section{Introduction}
\label{sec:introduction}

Arrow's impossibility theorem~\cite{Arrow50,Arrow63} is a classical result in social choice theory. There are $n$ voters who rank $m \ge 3$ candidates. Their votes are aggregated by a social choice function to produce a consensus ordering of the candidates. The theorem states that the only social choice function satisfying two natural properties, independence of irrelevant alternatives (IIA) and Pareto efficiency, is a dictatorship.

Concretely, in the case of $m = 3$ candidates $A,B,C$, we can represent voter $i$'s vote as a binary vector $(x_i,y_i,z_i) \in \{0,1\}^3$, where $x_i$ encodes whether the voter prefers $A$ or $B$ ($1$ or $0$, respectively), $y_i$ encodes whether they prefer $B$ or $C$, and $z_i$ encodes whether they prefer $C$ or $A$. Since $x_i,y_i,z_i$ correspond to an actual ordering, they must satisfy the following property:
\[
 (x_i,y_i,z_i) \in P_{\tNAE} := \{(0,0,1),(0,1,0),(0,1,1),(1,0,0),(1,0,1),(1,1,0)\}.
\]
The aggregated ordering should satisfy the same property. (Here NAE stands for Not-All-Equal.)

The IIA assumption states that the votes are aggregated ``item by item'', using three functions $f,g,h\colon\allowbreak \{0,1\}^n \to \{0,1\}$. If $x,y,z \in \{0,1\}^n$ represent the individual votes, then the aggregated ordering is $(f(x),g(y),h(z))$. Pareto efficiency states that if all voters agree on an item, then the aggregated ordering concurs. This means that $f(0,\dots,0) = g(0,\dots,0) = h(0,\dots,0) = 0$ and $f(1,\dots,1) = g(1,\dots,1) = h(1,\dots,1) = 1$.

We can now state Arrow's theorem in the case of three candidates:

\begin{theorem} \label{thm:arrow}
Suppose that $f,g,h\colon \{0,1\}^n \to \{0,1\}$ satisfy the following two conditions:
\begin{itemize}
    \item Polymorphicity:\footnote{In the literature on social choice theory, the functions $f,g,h$ are considered as a single function $(P_{\tNAE})^n \to P_{\tNAE}$ called an \emph{aggregator}.} Whenever $x,y,z \in \{0,1\}^n \to \{0,1\}$ are such that $(x_i,y_i,z_i) \in P_{\tNAE}$ for all $i$ then $(f(x),g(y),h(z)) \in P_{\tNAE}$.
    \item Pareto efficiency:\footnote{In the literature many other names are used, such as unanimity, idempotence, constancy, faithfulness, systematicity.} $f(b,\dots,b) = g(b,\dots,b) = h(b,\dots,b) = b$ for $b \in \{0, 1\}$.
\end{itemize}
Then there is a coordinate $j$ such that $f(x) = g(x) = h(x) = x_j$.
\end{theorem}

Rubinstein and Fishburn~\cite{RubinsteinFishburn86} asked what happens for other predicates $P$, possibly over non-binary domains. 
The first step in answering this question was taken by Nehring and Puppe~\cite{NehringPuppe02}, who determined all binary predicates for which all monotone polymorphisms are dictatorial. Dokow and Holzman~\cite{DokowHolzman10a} completely resolved the problem for binary predicates by determining all binary predicates for which all Pareto efficient polymorphisms are dictatorial (they called such predicates \emph{impossibility domains}). They also made progress on the problem for predicates over arbitrary finite alphabets~\cite{DokowHolzman10c}. Szegedy and Xu~\cite{SzegedyXu15} completely resolved the problem (for both supportiveness and unanimity) using methods of universal algebra. In contrast to the work of Dokow and Holzman, which gave explicit criteria, Szegedy and Xu showed that the case of arbitrary $n$ follows from an appropriate base case.
Other relevant results include~\cite{Wilson72,FishburnRubinstein86,ListPettit02,DietrichList07,DokowHolzman09,DokowHolzman10b,DietrichList13,Gibbard14}.

In universal algebra~\cite{Geiger68}, a function $f\colon \Sigma^n \to \Sigma$ is a polymorphism of a predicate $P \subseteq \Sigma^m$ if whenever $(x^{(1)}_i,\dots,x^{(m)}_i) \in P$ for all $i$, we have $(f(x^{(1)}), \dots, f(x^{(m)})) \in P$. This differs from the definition in \Cref{thm:arrow}, which allows different functions for different coordinates. We can recover the more general definition by assigning each coordinate a different \emph{sort}, and asking for \emph{multi-sorted} polymorphisms~\cite{BulatovJeavons03}. We adopt the term \emph{polymorphism}, which for us always signifies a multi-sorted polymorphism.

The universal algebra definition does not require the functions to be Pareto efficient. Mossel~\cite{Mossel09,Mossel12} proved a version of Arrow's theorem in this setting.\footnote{Mossel states his result for any number of candidates, but his more general result simply states that if we restrict attention to any three candidates, then the corresponding social choice functions behave as in the theorem.}

\begin{theorem}[Mossel] \label{thm:arrow-mossel}
Suppose that $f,g,h\colon \{0,1\}^n \to \{0,1\}$ are a polymorphism of $P_{\tNAE}$. Then one of the following cases holds:
\begin{itemize}
\item There exists $j$ such that $f(x) = g(x) = h(x) = x_j$.
\item There exists $j$ such that $f(x) = g(x) = h(x) = \bar{x}_j$ (i.e., $1 - x_j$).
\item One of $f,g,h$ is the constant~$0$ function, and another one is the constant~$1$ function.
\end{itemize}
\end{theorem}

In words, every polymorphism of $P_{\tNAE}$ is either dictatorial (all functions depend on a single coordinate) or a ``certificate'', meaning that the polymorphism fixes a subset of the coordinates to certain values which guarantee that $P_{\tNAE}$ is satisfied.

\medskip

In this paper, our goal is to prove an analog of the results of Dokow--Holzman and Szegedy--Xu in the setting of Mossel's result, that is, without assuming Pareto efficiency.

A formalization of our main results in \texttt{Lean4} can be found at \url{https://github.com/YuvalFilmus/Polymorphisms/}.

\subsection*{Setup}

\paragraph{Predicates} We will consider predicates $P \subseteq \Sigma_1 \times \cdots \times \Sigma_m$, where $\Sigma_1,\dots,\Sigma_m$ are finite sets of size at least~$2$. 

We assume that $P$ is \emph{non-degenerate}:
\begin{itemize}
    \item For each $i \in [m]$ and each $\sigma \in \Sigma_i$ there is $y \in P$ such that $y_i = \sigma$.

    Equivalently, the projection of $P$ to the $i$'th coordinate is $\Sigma_i$.

    \item $P$ depends on all coordinates, meaning that for all $i \in [m]$ there exist $y \in P$ and $z \notin P$ that differ only in the $i$'th coordinate.
\end{itemize}

\paragraph{Polymorphisms}

Functions $f_i\colon \Sigma_i^n \to \Sigma_i$ constitute a \emph{polymorphism} of $P$ if whenever $x^{(i)} \in \Sigma_i^n$ are such that $(x^{(1)}_j,\dots,x^{(m)}_j) \in P$ for all $j \in [n]$, then also $(f_1(x^{(1)}), \dots, f_m(x^{(m)})) \in P$.

\paragraph{Triviality}

Dokow and Holzman~\cite{DokowHolzman10a,DokowHolzman10c} define an $m$-ary predicate $P$ to be an \emph{impossibility domain} if whenever $f_1,\dots,f_m$ is a Pareto efficient polymorphism of $P$, there is a coordinate $j$ such that $f_i(x) = x_j$ for all $i \in [m]$. If we allow polymorphisms that are not Pareto efficient then we need to allow other ``trivial'' types of polymorphisms, such as certificates.

There are many possible definitions of trivial polymorphisms. We will use the following quite general one. Let $\Phi$ be a collection of tuples $(\phi_1,\dots,\phi_m)$, where $\phi_i\colon \Sigma_i \to \Sigma_i$.
A non-degenerate predicate $P$ is \emph{$\Phi$-trivial} if every $n$-ary polymorphism $f_1,\dots,f_m$ of $P$ has one of the following two types:
\begin{itemize}
    \item Dictatorial type: There exist $j \in [n]$ and $(\phi_1,\dots,\phi_m) \in \Phi$ such that $f_i(x) = \phi_i(x_j)$.
    \item Certificate type: There exists a subset $S \subseteq [m]$ and an assignment $\rho_i \in \Sigma_i$ for each $i \in S$ such that:
    \begin{itemize}
        \item $\rho$ is a \emph{certificate} for $P$: every $y \in \Sigma_1 \times \cdots \times \Sigma_m$ which agrees with $\rho$ on $S$ belongs to $P$.
        \item $f_1,\dots,f_m$ \emph{conform} to $\rho$: $f_i$ is the constant $\rho_i$ function for every $i \in S$.
    \end{itemize}
    We think of $\rho$ as a partial function, and so identify $S$ with its domain $\dom\rho$.
\end{itemize}

Here are several illustrative examples of the definition:
\begin{enumerate}
    \item $\Phi = \{ (\id_{\Sigma_1}, \dots, \id_{\Sigma_m}) \}$, where $\id_\Sigma$ is the identity function on $\Sigma$.

    In this case, the only dictators we allow are $f_1(x) = \cdots = f_m(x) = x_j$.
\end{enumerate}

For the remaining examples, we assume that all coordinates are binary: $\Sigma_i = \{0, 1\}$.

\begin{enumerate}[resume]
    \item $\Phi = \{ (\id,\dots,\id), (\neg,\dots,\neg) \}$, where $\neg(x) = \bar{x}$.

    In this case, we also allow dictators of the form $f_1(x) = \cdots = f_m(x) = \bar{x}_j$.

    \item $\Phi = \{ \id, \neg \}^m$.

    In this case, we allow dictators of the form $f_i(x) = x_j \oplus b_i$ for any $b \in \{0, 1\}^m$.

    \item $\Phi = \{ 0, 1, \id, \neg \}^m$.

    In this case, we allow dictators of the form $f_i(x) \in \{0, 1, x_j, \bar{x}_j\}$.
\end{enumerate}

We say that $P$ is $\Phi$-trivial \emph{for $n = n_0$} if the definition above holds for $n_0$-ary polymorphisms.

\subsection*{Results}
\label{sec:results}

Our main result is a reduction to the case $n = 2$, which echoes one of the main results of Szegedy and Xu~\cite{SzegedyXu15}.

\begin{restatable}{theorem}{thmmain} \label{thm:main}
A non-degenerate predicate $P$ is $\Phi$-trivial iff it is $\Phi$-trivial for $n = 2$.
\end{restatable}

We prove \Cref{thm:main} in \Cref{sec:main} by a simple induction.

\medskip

The case $n = 2$ could be hard to check by hand. In order to facilitate this, we prove a further reduction to the case $n = 1$, under an additional assumption on $\Phi$.

\begin{restatable}{theorem}{thmreduction} \label{thm:reduction}
Let $P \subseteq \Sigma_1 \times \cdots \times \Sigma_m$ be a non-degenerate predicate, where $|\Sigma_i| \geq 2$ for all $i$. Let $\Phi$ be such that for all $(\phi_1,\dots,\phi_m) \in \Phi$ and for all $i$, the function $\phi_i$ is a permutation of $\Sigma_i$.

If $P$ is $\Phi$-trivial for $n = 1$ then it is also $\Phi$-trivial for $n = 2$, unless one of the following cases holds:
\begin{enumerate}
    \item There is a coordinate $i \in [m]$ and $\sigma \in \Sigma_i$ such that $P$ is \emph{closed under setting $i$ to $\sigma$}, meaning that whenever $y \in P$ then also $y|_{i \gets \sigma} \in P$, where $y|_{i \gets \sigma}$ is obtained from $y$ by changing $y_i$ to $\sigma$.

    \item $P$ has a non-dictatorial \emph{AND/OR} polymorphism.

    An AND/OR polymorphism is a polymorphism $f_1,\dots,f_m$ where $f_i\colon \Sigma_i^2 \to \Sigma_i$ is as follows. If $|\Sigma_i| > 2$ then $f_i(x) = x_1$. If $|\Sigma_i| = 2$, without loss of generality $\Sigma_i = \{0, 1\}$, then either $f_i(x) = x_1 \land x_2$ or $f_i(x) = x_1 \lor x_2$.
    
    \item $P$ has a \emph{Latin square} polymorphism \emph{conforming} to $\Phi$.

    A Latin square polymorphism is a polymorphism $f_1,\dots,f_m$ where each $f_i\colon \Sigma_i^2 \to \Sigma_i$ is a Latin square: if we view it as a $\Sigma_i \times \Sigma_i$ matrix, then each row and each column is a permutation of $\Sigma_i$.

    The functions $f_1,\dots,f_m$ conform to $\Phi$ if for all $y \in P$ we have
    \[
     (f_1(y_1,\cdot),\dots,f_m(y_m,\cdot)), (f_1(\cdot,y_1),\dots,f_m(\cdot,y_m)) \in \Phi.
    \]
    Here $f_i(y_i,\cdot)$ is the function that takes $\sigma$ to $f_i(y_i,\sigma)$.
\end{enumerate}

Furthermore, if $\Sigma_1 = \cdots = \Sigma_m = \{0, 1\}$ and $P$ is not $\Phi$-trivial for $n = 1$ then this is witnessed either by a polymorphism $f_1,\dots,f_m\colon \{0,1\} \to \{0,1\}$ where $f_i(x) \in \{0, 1, x\}$ for all $i$, or by a polymorphism $f_1,\dots,f_m\colon \{0,1\} \to \{0,1\}$ where $f_i(x) \in \{x, \bar{x}\}$ for all $i$.
\end{restatable}

Let us demystify Latin square polymorphisms.
First, if $\Phi = \{(\id,\dots,\id)\}$ then there are no Latin square polymorphisms conforming to $\Phi$. Second, when $\Sigma_i = \{0, 1\}$, the only Latin square polymorphisms are $f_i(x) \in \{ x_1 \oplus x_2, x_1 \oplus x_2 \oplus 1 \}$. A predicate $P \subseteq \{0,1\}^m$ has such a polymorphism iff it is an affine subspace.

We prove \Cref{thm:reduction} in \Cref{sec:reduction}.

\medskip

As an application of the previous results, we study \emph{symmetric} predicates over $\{0,1\}$. These are predicates $P \subseteq \{0,1\}^m$ such that whether $y$ belongs to $P$ depends only on the Hamming weight of $y$. We consider two notions of triviality, corresponding to the following two collections:
\begin{itemize}
    \item $\Phi_{\id} = \{ (\id,\dots,\id) \}$.
    \item $\Phi_{\neg} = \{ \id, \neg \}^m$.
\end{itemize}

\begin{restatable}{theorem}{thmsymmetric} \label{thm:symmetric}
A non-degenerate symmetric predicate $P \subseteq \{0,1\}^m$ is $\Phi_{\neg}$-trivial if and only if $P$ is not one of the following predicates:
\begin{itemize}
    \item All vectors of even parity.
    \item All vectors of odd parity.
    \item All vectors of weight at least $w$, for some $w \in \{1,\dots,m-1\}$.
    \item All vectors of weight at least $w$ together with the all-zero vector, for some $w \in \{2,\dots,m\}$.
    \item All vectors of weight at most $w$, for some $w \in \{1,\dots,m-1\}$.
    \item All vectors of weight at most $w$ together with the all-one vector, for some $w \in \{0,\dots,m-2\}$.
\end{itemize}
Furthermore, if $P$ is $\Phi_{\neg}$-trivial then all dictatorial polymorphisms $f_1,\dots,f_m$ of $P$ are such that $f_1(x) = \cdots = f_m(x) = x_j$ or (possibly) $f_1(x) = \cdots = f_m(x) = \bar{x}_j$, for some $j \in [n]$.

\smallskip

The predicate $P$ is $\Phi_{\id}$-trivial if and only if $P$ is not one of the predicates listed above, and also $P$ is not closed under complementation (i.e., flipping all bits).
\end{restatable}

We prove \Cref{thm:symmetric} in \Cref{sec:symmetric}. 
\Cref{thm:arrow-mossel}, which states that the symmetric predicate $P_{\tNAE}$ is $\Phi_{\neg}$-trivial, immediately follows.

Using \Cref{thm:symmetric} as a starting point, we are able to determine, for any given symmetric predicate over $\{0,1\}$, all of its polymorphisms.

\begin{restatable}{theorem}{thmsymmetricclassification}
\label{thm:symmetric-classification}
Let $P \subseteq \{0, 1\}^m$ be a non-degenerate symmetric predicate.
\begin{enumerate}
\item \label{itm:sc-odd}
If $P = \{(0,1),(1,0)\}$ then $f_1,f_2$ is a polymorphism of $P$ iff $f_2(x) = \overline{f_1(\bar{x})}$ for all $x$.

\item \label{itm:sc-equal}
If $P = \{(0,\dots,0),(1,\dots,1)\}$ then $f_1,\dots,f_m$ is a polymorphism of $P$ iff $f_1 = \cdots = f_m$.

\item \label{itm:sc-parity}
If $m \ge 3$ and $P$ consists of all vectors of parity $b$ then $f_1,\dots,f_m\colon \{0,1\}^n \to \{0,1\}$ is a polymorphism of $P$ iff there exist a subset $J \subseteq [n]$ and bits $b_1,\dots,b_m \in \{0,1\}$ such that $f_i(x) = b_i \oplus \bigoplus_{j \in J} x_j$ for all $i$, where $b_1 \oplus \cdots \oplus b_m = (|J| + 1)b$ (here $(|J| + 1)b = 0$ if either $b = 0$ or $|J|$ is odd).

\item \label{itm:sc-atmost}
If $m \geq 3$ and $P$ consists of all vectors of weight at most $w$, where $1 \leq w \leq m-1$, then $f_1,\dots,f_m$ is a polymorphism of $P$ iff the corresponding families $F_1,\dots,F_m \subseteq 2^{[n]}$ are $(w+1)$-wise intersecting: if we choose $w+1$ of the families and one set from each family, then the intersection of the sets is non-empty.

If $m \geq 3$ and $P$ consists of all vectors of weight at least $m-w$, where $1 \leq w \leq m-1$, then an analogous condition to the preceding case holds, with $0$s and $1$s switched.

\item \label{itm:sc-atmost-m}
If $P$ consists of all vectors of weight at most $w$ together with $(1,\dots,1)$, where $1 \leq w \leq m-2$, then $f_1,\dots,f_m$ is a polymorphism of $P$ iff either $f_1 = \cdots = f_m$ and the common value is an AND of a (possibly empty) subset of coordinates, or at least $m-w$ of the functions are constant~$0$.

If $P$ consists of all vectors of weight at least $m-w$ together with $(0,\dots,0)$, where $1 \leq w \leq m-2$, then $f_1,\dots,f_m$ is a polymorphism of $P$ iff either $f_1 = \cdots = f_m$ and the common values is an OR of a (possibly empty) subset of coordinates, or at least $m-w$ of the functions are constant~$1$.

\item \label{itm:sc-trivial}
Suppose $P$ doesn't conform to any of these cases.

If $P$ is closed under complementation then $f_1,\dots,f_m$ is a polymorphism of $P$ iff either $f_1 = \cdots = f_m \in \{x_j, \bar{x}_j\}$ for some $j$, or $f_1,\dots,f_m$ are of certificate type.

If $P$ is not closed under complementation then $f_1,\dots,f_m$ is a polymorphism of $P$ iff either $f_1 = \cdots = f_m = x_j$ for some $j$, or $f_1,\dots,f_m$ are of certificate type.
\end{enumerate}
\end{restatable}
We prove \Cref{thm:symmetric-classification} in \Cref{sec:symmetric-classification}. Fortunately, the proof of the theorem is not much longer than its statement.

\medskip

Finally, we relate the notion of $\Phi$-triviality to the notion of impossibility domains studied by Dokow and Holzman~\cite{DokowHolzman10a,DokowHolzman10c}.

A non-degenerate predicate $P \subseteq \{0,1\}^m$ is an \emph{impossibility domain} according to Dokow and Holzman~\cite{DokowHolzman10a} if the only Pareto efficient polymorphisms of $P$ are dictators. Spelled out in full, $P$ is an impossibility domain if whenever $f_1,\dots,f_m\colon \{0,1\}^n \to \{0,1\}$ is a polymorphism of $P$ satisfying $f_i(b,\dots,b) = b$ for all $i \in [m]$ and all $b \in \{0,1\}$, then there exists $j \in [n]$ such that $f_1(x) = \cdots = f_m(x) = x_j$.

Dokow and Holzman~\cite{DokowHolzman10c} extended this definition to arbitrary alphabets. In fact, they gave two different definitions using two different extensions of Pareto efficiency to arbitrary alphabets, \emph{supportiveness} and \emph{unanimity}, both of which we define below. They focused on supportiveness. Later, Szegedy and Xu~\cite{SzegedyXu15} studied both definitions (they used the term \emph{idempotence} instead of unanimity).

A function $f\colon \Sigma^n \to \Sigma$ is \emph{unanimous} if $f(\sigma,\dots,\sigma) = \sigma$ for all $\sigma \in \Sigma$. It is \emph{supportive} if $f(\sigma_1,\dots,\sigma_n) \in \{\sigma_1,\dots,\sigma_n\}$ for all $\sigma_1,\dots,\sigma_n \in \Sigma$. We can use these two notions to define two notions of impossibility domains: impossibility domains with respect to supportiveness, and impossibility domains with respect to unanimity.
We relate the latter notion to $\Phi$-triviality.

\begin{restatable}{theorem}{thmDH} \label{thm:DH}
Let $P \subseteq \Sigma_1 \times \cdots \times \Sigma_m$ be a non-degenerate predicate, where $|\Sigma_i| \geq 2$ for all $i$. Let $\Phi$ be such that for all $(\phi_1,\dots,\phi_m) \in \Phi$, the function $\phi_i$ is a permutation of $\Sigma_i$.

Suppose that $P$ is $\Phi$-trivial for $n = 1$. Then $P$ is $\Phi$-trivial if and only if it is an impossibility domain with respect to unanimity.
\end{restatable}

We prove \Cref{thm:DH} in \Cref{sec:DH}.

\section*{Open questions}
\label{sec:open}

Post~\cite{Post20,Post42} classified all uni-sorted polymorphisms of predicates over $\{0, 1\}$. The case of larger alphabets is significantly more complicated~\cite{JanovMucnik59}, though some partial results were proven for $\{0, 1, 2\}$~\cite{Yablonski54,Zhuk15}. Classifying all multi-sorted polymorphisms appears daunting even for predicates over $\{0, 1\}$~\cite{BartoKapytka25}. Nevertheless, our results demonstrate that it is possible to determine which predicates support only simple polymorphisms.

Our notion of triviality is motivated by Mossel's extension of Arrow's theorem for three candidates, which shows that every polymorphism of $P_{\tNAE}$ is either of dictatorial type or of certificate type. When there are four or more candidates, a new kind of simple polymorphism arises. Abstracting away the details, define an $n$-ary polymorphism $f_1,\dots,f_m$ of $P \subseteq \Sigma_1 \times \cdots \times \Sigma_m$ to be of \emph{dictatorial-certificate type} if there exists a set $S \subseteq [m]$ such that:
\begin{itemize}
    \item For each $i \in S$, the function $f_i$ depends on at most one coordinate (not necessarily the same coordinate for all $i \in S$).

    \item Let $x^{(i)} \in \Sigma_i^n$ be such that $(x^{(1)}_j, \dots, x^{(m)}_j) \in P$ for all $j \in [n]$. Then $\rho := \bigl(f_i(x^{(i)})\bigr)|_{i \in S}$ is a certificate for $P$, meaning that if $y \in \Sigma_1 \times \cdots \times \Sigma_m$ agrees with $\rho$ on $S$ then $y \in P$.
\end{itemize}
It would be interesting to extend our results to this setting.

\smallskip

Another generalization is to the setting in which instead of a single predicate we have two predicates $P \subseteq \Sigma_1 \times \cdots \times \Sigma_m$ and $Q \subseteq \Delta_1 \times \cdots \times \Delta_m$. In this case functions $f_1,\dots,f_m$, where $f_i\colon \Sigma_i^n \to \Delta_i$, are a polymorphism if whenever $x^{(i)} \in \Sigma_i^n$ are such that $(x^{(1)}_j,\dots,x^{(m)}_j) \in P$ for all $j \in [n]$, then $(f_1(x^{(1)}), \dots, f_m(x^{(m)})) \in Q$.
This setting arose independently in universal algebra~\cite{Pippenger02,LPW18}, social choice theory~\cite{DokowHolzman10b}, and complexity theory~\cite{AGH17,BG16,BG21,BBKO21}.

The simplest setting is when $\Delta_i = \Sigma_i$ and $P \subseteq Q$, and the notion of simplicity depends on the context. Dokow and Holzman~\cite{DokowHolzman10b} describe a particular scenario arising from social choice theory, and determine which predicates over $\{0,1\}$ are impossibility domains in their setting. It would be interesting to extend their results to our setting, as well as to larger alphabets.

\smallskip

Other interesting open questions include removing the assumption on $\Phi$ from \Cref{thm:reduction,thm:DH}, extending \Cref{thm:DH} to supportiveness, and generalizing \Cref{thm:arrow} to arbitrary finite alphabets.

\section{Main result}
\label{sec:main}

In this section we prove our main result:

\thmmain*

Let $P \subseteq \Sigma_1 \times \cdots \times \Sigma_m$ be a non-degenerate predicate. If $P$ is $\Phi$-trivial then it is clearly $\Phi$-trivial for $n = 2$.

For the other direction, we first observe that $P$ is $\Phi$-trivial for $n=0$ and $n=1$. The case $n=0$ is trivial, since all polymorphisms are constant in this case, and so they conform to a certificate. In the case $n=1$, we can extend the given unary polymorphism $f_1,\dots,f_m$ to a binary polymorphism $F_1,\dots,F_m$ in which each function depends only on the first argument. Applying the case $n=2$, there are three cases:
\begin{itemize}
    \item There exists $(\phi_1,\dots,\phi_m) \in \Phi$ such that $F_i(x) = \phi_i(x_1)$ for all $i$.

    In this case $f_i(x) = \phi_i(x)$, as needed.

    \item There exists $(\phi_1,\dots,\phi_m) \in \Phi$ such that $F_i(x) = \phi_i(x_2)$ for all $i$.

    Since $F_i$ doesn't depend on $x_2$, in this case all $f_i$ are constant, and so they conform to a certificate.

    \item The functions $F_1,\dots,F_m$ conform to a certificate $\rho$.

    In this case $f_1,\dots,f_m$ also conform to $\rho$.
\end{itemize}

We prove that $P$ is $\Phi$-trivial for all $n > 2$ by induction. Assuming that $P$ is $\Phi$-trivial for $n = 2$ and for a given value of $n \ge 2$, we prove that it is also $\Phi$-trivial for $n + 1$.

From here on, we assume that we are given an $(n+1)$-ary polymorphism $f_1,\dots,f_m$. Our goal is to show that it is either of dictatorial type or of certificate type.

The idea of the proof is to consider the functions $f_i|_\sigma$ obtained by fixing the final argument to $\sigma \in \Sigma_i$. For each $y \in P$, the functions $f_1|_{y_1},\dots,f_m|_{y_m}$ are an $n$-ary polymorphism of $P$, and we can apply the inductive hypothesis to them. In order to complete the proof, we need to aggregate the structure of the $f_i|_\sigma$ to conclude a structure of the $f_i$. This will be accomplished by applying the $\Phi$-triviality of $2$-ary polymorphisms to a specially constructed function which abstracts the salient structure of the $f_i|_\sigma$ in a usable way, as described by the following lemma.

\begin{lemma} \label{lem:main-g}
There are functions $g_1,\dots,g_m$, where $g_i\colon \Sigma_i^2 \to \Sigma_i$, such that the following properties hold:
\begin{enumerate}
    \item $g_1,\dots,g_m$ is a polymorphism of $P$.
    \label{itm:main-g-poly}
    \item If $g_i(x) = \phi(x_1)$ for some $\phi\colon \Sigma_i \to \Sigma_i$ then $f_i(x) = \phi(x_{n+1})$.
    \label{itm:main-g-constant}
    \item If $g_i(x) = \phi(x_2)$ for some $\phi\colon \Sigma_i \to \Sigma_i$ then there are coordinates $s(i,\sigma) \in [n]$ such that $f_i|_\sigma(x) = \phi(x_{s(i,\sigma)})$.
    \label{itm:main-g-dictator}
\end{enumerate}
\end{lemma}
\begin{proof}
We first define for every $i$ a function $h_i$ which takes as input $\sigma \in \Sigma_i$ and returns either a function $\Sigma_i \to \Sigma_i$ or $\bot$. The definition is as follows: if $f_i|_\sigma(x) = \phi(x_s)$ for some (possibly constant) $\phi\colon \Sigma \to \Sigma$ then $h_i(\sigma) = \phi$, and otherwise $h_i(\sigma) = \bot$.

We can now define $g_i$, considering two cases:
\begin{itemize}
    \item There exists $\sigma_0 \in \Sigma_i$ such that $h_i(\sigma_0) \neq \bot$.

    Let $\psi\colon \Sigma_i \to \Sigma_i$ be any non-constant function which is different from $h_i(\sigma_0)$; such a function exists since $|\Sigma_i| \geq 2$, and so there are at least two non-constant functions on $\Sigma_i$.

    We define $g_i(\sigma,a) = h_i(\sigma)(a)$ if $h_i(\sigma) \neq \bot$, and $g_i(\sigma,a) = \psi(a)$ otherwise.

    \item $h_i(\sigma) = \bot$ for all $\sigma \in \Sigma_i$.

    Let $\psi',\psi''\colon \Sigma_i \to \Sigma_i$ be two different non-constant functions. Single out some $\sigma_0 \in \Sigma_i$. We let $g_i(\sigma_0,a) = \psi'(a)$ and $g_i(\sigma,a) = \psi''(a)$ for $\sigma \neq \sigma_0$.
\end{itemize}

Let us now verify the stated properties one by one:
\begin{enumerate}
    \item Suppose that $y,z \in P$. We need to show that $(g_1(y_1,z_1), \dots, g_m(y_m,z_m)) \in P$.

    Since $y \in P$, the functions $f_1|_{y_1},\dots,f_m|_{y_m}$ are a polymorphism of $P$. We now consider two cases, according to whether this polymorphism is of dictatorial type or of certificate type.

    Suppose first that $f_1|_{y_1},\dots,f_m|_{y_m}$ is of dictatorial type: there exist $j \in [n]$ and $(\phi_1,\dots,\phi_m) \in \Phi$ such that $f_i|_{y_i}(x) = \phi_i(x_j)$. Then $h_i(y_i) = \phi_i$, and so $g_i(y_i,a) = \phi_i(a)$.
    
    Define vectors $x^{(i)} \in \Sigma_i^{n+1}$ as follows: $x^{(i)}_1 = \cdots = x^{(i)}_n = z_i$ and $x^{(i)}_{n+1} = y_i$. By construction, $(x^{(1)}_j,\dots,x^{(m)}_j) \in P$ for all $j$, and so $(f_1(x^{(1)}),\dots,f_m(x^{(m)})) \in P$. Now $f_i(x^{(i)}) = f_i|_{y_i}(z_i,\dots,z_i) = \phi_i(z_i) = g_i(y_i,z_i)$, and so $(g_1(y_1,z_1), \dots, g_m(y_m,z_m)) \in P$.

    Suppose next that $f_1|_{y_1},\dots,f_m|_{y_m}$ conform to some certificate $\rho$. For every $i \in \dom\rho$, the function $f_i|_{y_i}$ is the constant $\rho_i$ function. Hence $h_i(y_i)$ is the constant $\rho_i$ function, and so $g_i(y_i,z_i) = \rho_i$. Since $\rho$ is a certificate, this shows that $(g_1(y_1,z_1), \dots, g_m(y_m,z_m)) \in P$.

    \item Suppose that $g_i(x) = \phi(x_1)$ for some $\phi\colon \Sigma_i \to \Sigma_i$. We need to show that $f_i(x) = \phi(x_{n+1})$.

    The definition of $g_i$ shows that $h_i(\sigma) \neq \bot$ for all $\sigma$ (since otherwise the function $a \mapsto g_i(\sigma,a)$ would not be constant), and so $h_i(\sigma)$ is the constant $\phi(\sigma)$ function. This implies that $f_i|_\sigma$ is the constant $\phi(\sigma)$ function, and so $f_i(x) = \phi(x_{n+1})$.

    \item Suppose that $g_i(x) = \phi(x_2)$ for some function $\phi\colon \Sigma_i \to \Sigma_i$. We need to show that for all $\sigma$ we have $f_i|_\sigma(x) = \phi(x_s)$ for some $s \in [n]$ which could depend on $\sigma$.

    If $h_i(\sigma) = \bot$ for some $\sigma \in \Sigma_i$ then the definition of $g_i$ guarantees that the functions $a \mapsto g_i(\sigma,a)$ are not all the same, and so it is not the case that $g_i(x)$ depends only on $x_2$. Hence $h_i(\sigma) \neq \bot$ for all $\sigma \in \Sigma_i$, implying that $g_i(x) = h_i(x_1)(x_2)$. Thus $h_i(\sigma) = \phi$ for all $\sigma$, and so for each $\sigma$ we have $f_i|_\sigma(x) = \phi(x_s)$ for some $s \in [n]$. \qedhere
\end{enumerate}
\end{proof}

Since $g_1,\dots,g_m$ is a polymorphism of $P$ and $P$ is $\Phi$-trivial for $n = 2$, we conclude that $g_1,\dots,g_m$ is either of dictatorial type or of certificate type. In order to complete the proof of \Cref{thm:main}, we consider three cases:
\begin{enumerate}
    \item All functions $g_1,\dots,g_m$ depend on $x_1$.
    \item All functions $g_1,\dots,g_m$ depend on $x_2$.
    \item All functions $g_1,\dots,g_m$ are of certificate type.
\end{enumerate}

\paragraph{There exists $(\phi_1,\dots,\phi_m) \in \Phi$ such that $g_i(x) = \phi_i(x_1)$ for all $i$.}

Applying \Cref{itm:main-g-constant} of \Cref{lem:main-g}, we see that $f_i|_\sigma(x) = \phi_i(\sigma)$ for all $i,\sigma$, and so $f_i(x) = \phi_i(x_{n+1})$ for all $i$. Therefore $f_1,\dots,f_m$ is of dictator type, where the dictator is $x_{n+1}$.

\paragraph{There exists $(\phi_1,\dots,\phi_m) \in \Phi$ such that $g_i(x) = \phi_i(x_2)$ for all $i$.}

Applying \Cref{itm:main-g-dictator} of \Cref{lem:main-g}, we conclude that there are coordinates $s(i,\sigma) \in [n]$ such that $f_i|_\sigma(x) = \phi_i(x_{s(i,\sigma)})$.

Let $A \subseteq [m]$ consist of those coordinates for which $\phi_i$ is not constant. If there exists $s \in [n]$ such that $s(i,\sigma) = s$ whenever $i \in A$ then we can set $s(i,\sigma) = s$ for $i \notin A$ to obtain that $f_1,\dots,f_m$ are of dictatorial type, where the dictator is $x_s$. So suppose that $\{ s(i,\sigma) : i \in A, \sigma \in \Sigma_i \}$ contains at least two different coordinates; in particular, $A$ is non-empty.

Recall that for every $y \in P$, the $n$-ary functions $f_1|_{y_1},\dots,f_m|_{y_m}$ are a polymorphism of $P$. Suppose first that there is some $y \in P$ such that $s(i',y_{i'}) \neq s(i'',y_{i''})$ for some $i',i'' \in A$.
This implies that $f_1|_{y_1},\dots,f_m|_{y_m}$ cannot be of dictatorial type.
Applying the induction hypothesis, we see that $f_1|_{y_1},\dots,f_m|_{y_m}$ must conform to some certificate $\rho$, where necessarily $\dom\rho \subseteq \bar{A}$. In this case, the functions $f_1,\dots,f_m$ also conform to $\rho$.

Suppose next that for every $y \in P$ there exists $s_y$ such that $s(i,y_i) = s_y$ for all $i \in A$. We would like to show that in this case as well $f_1,\dots,f_m$ are of certificate type.

Single out some arbitrary $y_0 \in P$ and let $s_0 = s_{y_0}$. For $i \in A$, we partition $\Sigma_i$ into two parts $\Sigma_{i,0},\Sigma_{i,1}$, where $\Sigma_{i,0}$ consists of those $\sigma \in \Sigma_i$ such that $s(i,\sigma) = s_0$. Thus for every $y \in P$, either $y_i \in \Sigma_{i,0}$ for all $i \in A$, or $y_i \in \Sigma_{i,1}$ for all $i \in A$. We will use this to construct a $2$-ary polymorphism $\chi_1,\dots,\chi_m$ which is not of dictatorial type, and deduce that $f_1,\dots,f_m$ are of certificate type.

The definition of $\chi_i$ is quite simple. If $i \notin A$, we define $\chi_i(x) = \phi_i$. If $i \in A$, we define $\chi_i(x) = \phi_i(x_1)$ if $x_1 \in \Sigma_{i,0}$ and $\chi_i(x) = \phi_i(x_2)$ if $x_1 \in \Sigma_{i,1}$.
Since $g_1,\dots,g_m$ are a polymorphism of $P$, so are $\phi_1,\dots,\phi_m$. This immediately implies that $\chi_1,\dots,\chi_m$ is a polymorphism.

Applying the case $n = 2$, we deduce that $\chi_1,\dots,\chi_m$ are either of dictatorial type or of certificate type. We claim that they cannot be of dictatorial type. It is clear that they cannot depend only on $x_1$. If they depended only on $x_2$, then considering $y_0$ as the first argument, we see that $\chi_1,\dots,\chi_m$ have to be constant, contradicting the non-emptiness of $A$.

Thus $\chi_1,\dots,\chi_m$ must conform to some certificate $\rho$. By construction, $\dom\rho \subseteq \bar{A}$ and $\rho_i = \phi_i$ for all $i \in \dom\rho$. Hence $f_1,\dots,f_m$ also conform to $\rho$.

\paragraph{The functions $g_1,\dots,g_m$ conform to some certificate $\rho$.}

Applying \Cref{itm:main-g-constant} of \Cref{lem:main-g}, we see that $f_i|_\sigma$ is the constant $\rho_i$ function whenever $i \in \dom\rho$, and so $f_i$ is the constant $\rho_i$ function whenever $i \in \dom\rho$. Hence $f_1,\dots,f_m$ also conform to $\rho$.

\section{Reduction to \texorpdfstring{$n=1$}{n=1}}
\label{sec:reduction}

In this section we prove the reduction from $n=2$ to $n=1$:

\thmreduction*

Suppose that $P$ is $\Phi$-trivial for $n = 1$, and let $f_1,\dots,f_m$ be a $2$-ary polymorphism of $P$. We will attempt to show that $f_1,\dots,f_m$ are of dictatorial type or of certificate type. The proof will fail in certain cases, and each of these cases will be covered by one of the cases in the statement of the theorem. Later on, we will prove the ``furthermore'' part.

As in the proof of \Cref{thm:main}, we define $f_i|_\sigma$ to be the function obtained by fixing the final argument to $\sigma \in \Sigma_i$. Thus for every $y \in P$, the functions $f_1|_{y_1},\dots,f_m|_{y_m}$ are a $1$-ary polymorphism of $P$.

As in the proof of \Cref{thm:main}, we capitalize on this observation by considering auxiliary functions $h_i\colon \Sigma_i \to \Sigma_i \cup \{*\}$ (where $*$ is a symbol not in $\Sigma_i$) which abstract the situation: if $f_i|_\sigma$ is the constant $\tau$ function then we define $h_i(\sigma) = \tau$, and otherwise we define $h_i(\sigma) = *$. The idea is that in certain cases we can fill in the stars to obtain a $1$-ary polymorphism of $P$, gaining insight on $f_1,\dots,f_m$ by applying $\Phi$-triviality for $n = 1$.

We start with a simple observation which follows immediately from the observation that $f_1|_{y_1},\dots,f_m|_{y_m}$ is a polymorphism for every $y \in P$, coupled with $\Phi$-triviality for $n = 1$.

\begin{lemma} \label{lem:reduction-h}
For every $y \in P$, one of the following cases holds:
\begin{enumerate}
    \item Dictatorial case: $h_i(y_i) = *$ for all $i$.

    In this case there exists $(\phi_1,\dots,\phi_m) \in \Phi$ such that $f_i|_{y_i} = \phi_i$ for all $i$.

    \item Certificate case: There is a certificate $\rho$ such that $h_i(y_i) = \rho_i$ for all $i \in \dom\rho$.
\end{enumerate}
\end{lemma}

We now consider several cases:

\begin{enumerate}
    \item The certificate case holds for all $y \in P$.
    \item There is $\eta \in P$ for which the dictatorial case holds, and furthermore, $h_{i_0}(\sigma_0) \neq *$ for some $i_0,\sigma_0$.
    \item $h_i(\sigma) = *$ for all $i,\sigma$.
\end{enumerate}

Each case will involve a different argument. The first two cases will involve various ways of \emph{completing} each $h_i$ to a function $g_i\colon \Sigma_i \to \Sigma_i$ so that $g_1,\dots,g_m$ is a polymorphism of $P$. This means that $g_i(\sigma) = h_i(\sigma)$ whenever $h_i(\sigma) \neq *$.

\subsection*{The certificate case holds for all $y \in P$}

We start with the completion process.

\begin{lemma} \label{lem:reduction-g1}
We can complete $h_1,\dots,h_m$ to a polymorphism $g_1,\dots,g_m$ such that for each $i$, if $g_i$ is the constant $\tau$ function then $h_i$ is also the constant $\tau$ function.
\end{lemma}
\begin{proof}
The definition of $g_i$ is quite simple. If $h_i(\sigma) = *$ for all $\sigma$ then we define $g_i(\sigma) = \sigma$ for all $\sigma$. Otherwise, suppose $h_i(\sigma) \neq *$. Since $|\Sigma_i| \geq 2$, we can find $\tau \in \Sigma_i$ which is different from $h_i(\sigma)$. We let $g_i(\sigma) = h_i(\sigma)$ if $h_i(\sigma) \neq *$, and $g_i(\sigma) = \tau$ otherwise.

The construction guarantees that $g_i$ can be constant only if $g_i = h_i$. It remains to check that $g_1,\dots,g_m$ is a polymorphism of $P$.
Let $y \in P$. Since the certificate case holds for all $y \in P$, there is a certificate $\rho$ such that $g_i(y_i) = h_i(y_i) = \rho_i$ for all $i \in \dom\rho$. Since $\rho$ is a certificate, this means that $(g_1(y_1),\dots,g_m(y_m)) \in P$.
\end{proof}

Since $g_1,\dots,g_m$ is a $1$-ary polymorphism of $P$ and $P$ is $\Phi$-trivial for $n = 1$, the functions $g_1,\dots,g_m$ are either of dictatorial type or of certificate type. We consider the two cases separately.

\paragraph{There exists $(\phi_1,\dots,\phi_m) \in \Phi$ such that $g_i = \phi_i$ for all $i$.}

We first observe that $\phi_1^{-1},\dots,\phi_m^{-1}$ is also a polymorphism of $P$. Indeed, $\phi_1^r,\dots,\phi_m^r$ is a polymorphism of $P$ for every $r \geq 0$, and if we take $r = \prod_{i=1}^m |\Sigma_i|! - 1$, then $\phi_i^r = \phi_i^{-1}$.

\smallskip

Suppose first that $h_{i_0}(\sigma_0) = *$ for some $i_0,\sigma_0$. Consider any $y \in P$ satisfying $y_{i_0} = \sigma_0$. By assumption, there is a certificate $\rho$ such that $\rho_i = h_i(y_i) = g_i(y_i) = \phi_i(y_i)$ for all $i \in \dom\rho$. Since $\phi_1^{-1},\dots,\phi_m^{-1}$ is a polymorphism of $P$, the assignment $\lambda$ defined by $\lambda_i = \phi_i^{-1}(\rho_i) = y_i$ for all $i \in \dom\rho$ is also a certificate. Observe that $i_0 \notin \dom\rho$. This means that $y|_{i_0 \gets \tau} \in P$ for all $\tau \in \Sigma_{i_0}$. Let us record this:
\begin{equation} \label{eq:reduction-1}
y \in P \text{ and } y_{i_0} = \sigma_0 \Longrightarrow y|_{i_0 \gets \tau} \in P \text{ for all } \tau.
\end{equation}

We now consider two cases: $|\Sigma_{i_0}| = 2$ and $|\Sigma_{i_0}| > 2$. If $|\Sigma_{i_0}| = 2$ then $P$ is closed under setting $i_0$ to $\bar{\sigma}_0$, which is one of the cases in the statement of the theorem. Indeed, if $y \in P$ satisfies $y_{i_0} = \bar{\sigma}_0$ then $y|_{i_0 \gets \bar{\sigma}_0} = y \in P$, and otherwise $y_{i_0} = \sigma_0$, and so \Cref{eq:reduction-1} shows that $y_{i_0 \gets \bar{\sigma}_0} \in P$.

The other case, $|\Sigma_{i_0}| > 2$, contradicts the assumption that $P$ is $\Phi$-trivial for $n = 1$. To see this, consider the functions $e_1,\dots,e_m$ defined as follows. If $i \neq i_0$ then $e_i = \id$. We let $e_{i_0}(\sigma) = \sigma$ for $\sigma \neq \sigma_0$, and $e_{i_0}(\sigma_0) = \sigma_1$ for some $\sigma_1 \neq \sigma_0$.
\Cref{eq:reduction-1} implies that $e_1,\dots,e_m$ are a polymorphism of $P$, and so are of either dictatorial type or certificate type.
However, by construction, no $e_i$ is constant, and so $e_1,\dots,e_m$ cannot be of certificate type; and $e_{i_0}$ is not a permutation, and so $e_1,\dots,e_m$ cannot be of dictatorial type.

\smallskip

Finally, suppose that $h_i(\sigma) \neq *$ for all $i,\sigma$. In this case $f_i|_\sigma$ is the constant $h_i(\sigma)$ function, and so $f_i(x) = h_i(x_2) = g_i(x_2) = \phi_i(x_2)$. Hence $f_1,\dots,f_m$ is of dictatorial type.

\paragraph{The functions $g_1,\dots,g_m$ conform to some certificate $\rho$.}

\Cref{lem:reduction-g1} implies that $h_i$ is the constant $\rho_i$ function for all $i \in \dom\rho$. The definition of $h_1,\dots,h_m$ implies that $f_i$ is the constant $\rho_i$ function for all $i \in \dom\rho$. Hence $f_1,\dots,f_m$ also conform to $\rho$.

\subsection*{There is $\eta \in P$ for which the dictatorial case holds, and furthermore, $h_{i_0}(\sigma_0) \neq *$ for some $i_0,\sigma_0$}

Rephrasing the first assumption, $h_i(\eta_i) = *$ for all $i$.
We start with the completion process.

\begin{lemma} \label{lem:reduction-g2}
Let $z \in P$. Consider the completion $g_1^z,\dots,g_m^z$ of $h_1,\dots,h_m$ defined as follows: $g_i^z(\sigma) = h_i^z(\sigma)$ if $h_i^z(\sigma) \neq *$, and $g_i^z(\sigma) = z_i$ if $h_i^z(\sigma) = *$.

The functions $g_1^z,\dots,g_m^z$ are a polymorphism of $P$.
\end{lemma}
\begin{proof}
Let $y \in P$. If $h_i(y_i) = *$ for all $i$ then $(g_1^z(y_1),\dots,g_m^z(y_m)) = z \in P$. Otherwise, invoking \Cref{lem:reduction-h}, there is a certificate $\rho$ such that $g_i^z(y_i) = h_i(y_i) = \rho_i$ for all $i \in \dom\rho$. Hence $(g_1^z(y_1),\dots,g_m^z(y_m)) \in P$ since $\rho$ is a certificate.
\end{proof}

Invoking the lemma, $g_1^\eta,\dots,g_m^\eta$ are a $1$-ary polymorphism of $P$, and so either of dictatorial type or of certificate type. We consider the two cases separately.

\paragraph{There exists $(\phi_1,\dots,\phi_m) \in \Phi$ such that $g_i^\eta = \phi_i$ for all $i$.}

Since each $\phi_i$ is a permutation, this can only happen if $h_i(\sigma) \neq *$ for all $\sigma \neq \eta_i$.

Consider any $\zeta \in P$ which is different from $\eta$. Invoking \Cref{lem:reduction-g2} again, the functions $g_1^\zeta,\dots,g_m^\zeta$ are a polymorphism of $P$, and so either of dictatorial type or of certificate type. Since $\zeta \neq \eta$, there must be a coordinate $i_1$ such that $\zeta_{i_1} \neq \eta_{i_1}$. Since $g_{i_1}^\eta$ is a permutation and $g_{i_1}^\zeta$ differs from $g_{i_1}^\eta$ only on input $\eta_{i_1}$, we see that $g_{i_1}^\zeta$ cannot be a permutation. Hence $g_1^\zeta,\dots,g_m^\zeta$ must conform to some certificate $\rho$.

If $i \in \dom\rho$ then $g_i^\zeta$ must be constant. Since $g_i^\eta$ is a permutation and $g_i^\zeta$ differs from it only on input $\eta_i$, necessarily $|\Sigma_i| = 2$ and $\zeta_i = \bar{\eta}_i$.
Letting $B = \{ i : |\Sigma_i| = 2 \}$, this shows that: \begin{equation} \label{eq:reduction-2}
\zeta \in P \text{ and } \zeta \neq \eta \Longrightarrow w \in P \text{ whenever } w_i = \zeta_i \text{ for all } i \in B \text{ such that } \zeta_i = \bar{\eta}_i.
\end{equation}

This property allows us to construct a polymorphism $e_1,\dots,e_m$ of AND/OR type, which is one of the cases in the statement of the theorem. If $i \notin B$, we let $e_i(x) = x_1$. If $i \in B$ then $e_i(x) = \eta_i$ if $x_1 = x_2 = \eta_i$, and $e_i(x) = \bar{\eta}_i$ otherwise. When $\Sigma_i = \{0, 1\}$, the function $e_i$ is the OR function if $\eta_i = 0$ and the AND function if $\eta_i = 1$.

To see that this is a polymorphism, let $y,z \in P$. If $y = z = \eta$ then $(e_1(y_1,z_1),\dots,e_m(y_m,z_m)) = \eta \in P$. Otherwise, suppose without loss of generality that $z \neq \eta$. Applying \Cref{eq:reduction-2} with $\zeta = z$, it suffices to show that $e_i(y_i,z_i) = z_i$ for all $i \in B$ such that $z_i = \bar{\eta}_i$. This follows directly from the definition of $e_i$.

To see that $e$ is non-dictatorial, take any $\zeta \in P$ other than $\eta$. The argument about shows that $g_1^\zeta,\dots,g_m^\zeta$ must conform to some certificate $\rho$, and furthermore, if $i \in \dom\rho$ then $i \in B$, and so $e_i$ is not a dictator.

\paragraph{The functions $g_1^\eta,\dots,g_m^\eta$ conform to some certificate $\rho$.}

Suppose first that $\tau := h_{i_1}(\sigma_1) \notin \{ \eta_{i_1}, * \}$ for some $i_1,\sigma_1$. This means that $g_{i_1}^\eta$ is not constant, and so $i_1 \notin \dom\rho$. Thus $\zeta := \eta|_{i_1 \gets \sigma_1} \in P$. Observe that $h_{i_1}(\zeta_{i_1}) = \tau \neq *$ whereas $h_i(\zeta_i) = h_i(\eta_i) = *$ for all $i \neq i_1$. In view of \Cref{lem:reduction-h}, this shows that $P$ contains all $z$ such that $z_{i_1} = \tau$. In particular, $P$ is closed under setting $i_1$ to $\tau$, which is one of the cases in the statement of the theorem.

\medskip

We can thus assume that $h_i(\sigma) \in \{\eta_i,*\}$ for all $i,\sigma$. Let $C$ consist of $i$ such that $h_i(\sigma) = *$ for all $\sigma$. By assumption, $i_0 \notin C$.

\smallskip

Suppose first that $C \neq \emptyset$ or $|\Sigma_i| > 2$ for some $i$. Consider any $\zeta \in P$. Invoking \Cref{lem:reduction-g2}, the functions $g_1^\zeta,\dots,g_m^\zeta$ are a polymorphism, and so either of dictatorial type or of certificate type. We claim that one of these functions is not a permutation, and so $g_1^\zeta,\dots,g_m^\zeta$ must be of certificate type.

Indeed, if $C \neq \emptyset$, say $i \in C$, then $g_i^\zeta$ is constant. Similarly, if $|\Sigma_i| > 2$ then $g_i^\zeta$ cannot be a permutation: either $h_i^\zeta$ has at least two $*$-inputs, both of which are set to $\zeta_i$ in $g_i^\zeta$; or it has at least two $\eta_i$-inputs.

Thus $g_1^\zeta,\dots,g_m^\zeta$ conform to some certificate $\rho^\zeta$. If $i\in\dom\rho^\zeta$ then $g_i^\zeta$ is constant, and so either $i \in C$ or $\zeta_i = \eta_i$; in both cases, $\rho^\zeta_i = \zeta_i$.
In particular, if $\zeta_{i_0} \neq \eta_{i_0}$ then $i_0 \notin \dom\rho^\zeta$. Since $\rho^\zeta$ agrees with $\zeta$ on its domain, this implies that $\zeta|_{i_0 \gets \eta_{i_0}} \in P$. Therefore $P$ is closed under setting $i_0$ to $\eta_{i_0}$, which is one of the cases in the theorem.

\smallskip

Suppose now that $C = \emptyset$ and $|\Sigma_i| = 2$ for all $i$. This means that $h_i(\bar{\eta}_i) = \eta_i$ for all $i$.

Consider any $\zeta \in P$. As before, the functions $g_1^\zeta,\dots,g_m^\zeta$ are a polymorphism and so either of dictatorial type or of certificate type.
If $g_1^\zeta,\dots,g_m^\zeta$ are of dictatorial type then $g_i^\zeta$ is a permutation for all $i$, and so $\zeta_i = \bar{\eta}_i$ for all $i$; in short, $\zeta = \bar{\eta}$. In all other cases, $g_1^\zeta,\dots,g_m^\zeta$ is of certificate type.

Suppose therefore that $\zeta \in P$ is different from $\bar{\eta}$. Then $g_1^\zeta,\dots,g_m^\zeta$ conform to a certificate $\rho^\zeta$. If $i\in\dom\rho^\zeta$ then $\zeta_i = \eta_i$. Equivalently, if $\zeta_i \neq \eta_i$ then $i \notin \dom\rho^\zeta$. Since $\rho^\zeta$ agrees with $\zeta$ on its domain, this shows that:
\begin{equation} \label{eq:reduction-3}
\zeta \in P \text{ and } \zeta \neq \bar{\eta} \Longrightarrow \zeta|_{i \gets \eta_i} \in P \text{ for all } i.
\end{equation}

This property allows us to construct a polymorphism $e_1,\dots,e_m$ of AND/OR type, which is one of the cases in the statement of the theorem: $e_i(x) = \bar{\eta}_i$ if $x_1 = x_2 = \bar{\eta}_i$, and $e_i(x) = \eta_i$ otherwise. To see that this is indeed a polymorphism, let $y,z \in P$. If $y = z = \bar{\eta}$ then $(e_1(y_1,z_1),\dots,e_m(y_m,z_m)) = \bar{\eta} \in P$. Otherwise, suppose without loss of generality that $z \neq \bar{\eta}$. Observe that $(e_1(y_1,z_1),\dots,e_m(y_m,z_m))$ is obtained from $z$ be setting to $\eta_i$ coordinates $i$ such that $y_i = \eta_i$. Therefore \Cref{eq:reduction-3} implies that $(e_1(y_1,z_1),\dots,e_m(y_m,z_m)) \in P$.

\subsection*{For all $i,\sigma$ we have $h_i(\sigma) = *$}

The dictatorial case of \Cref{lem:reduction-h} applies for all $y \in P$. Since $P$ is non-degenerate, for every $i,\sigma$ we can find $y \in P$ such that $y_i = \sigma$, and so $f_i|_\sigma$ is a permutation.

So far we have restricted $f_1,\dots,f_m$ to a $1$-ary polymorphism according to the second argument. We can do the same, but according to the first argument. Running the argument so far, one of the following happens:
\begin{itemize}
    \item The argument shows that $f_1,\dots,f_m$ is of dictatorial type or of certificate type.
    \item The argument shows that one of the cases in the statement of the theorem holds.
    \item The argument reaches the current case ($h_i(\sigma) = *$ for all $i,\sigma$).
    In this case we immediately conclude that $f_1,\dots,f_m$ is a Latin square polymorphism conforming to $\Phi$.
\end{itemize}

\subsection*{Furthermore part}

Suppose that $\Sigma_1 = \cdots = \Sigma_m = \{0, 1\}$ and $P$ is not $\Phi$-trivial for $n = 1$. Then there is a $1$-ary polymorphism $f_1,\dots,f_m$ of $P$ which is neither of dictatorial type nor of certificate type.

Since $\Sigma_1 = \cdots = \Sigma_m = \{0, 1\}$, each $f_i$ is one of the functions $0,1,x,\bar{x}$. Let $g_i(x) = f_i(f_i(x)) \in \{0,1,x\}$. Clearly $g_1,\dots,g_m$ is a polymorphism of $P$. We now consider two cases, according to whether $g_1,\dots,g_m$ are of dictatorial type or not.

If $g_1,\dots,g_m$ are of dictatorial type then $g_1 = \cdots = g_m = x$ and so $f_1,\dots,f_m \in \{x, \bar{x}\}$, hence $f_1,\dots,f_m$ are the claimed polymorphism.

Suppose next that $g_1,\dots,g_m$ are not of dictatorial type. They cannot conform to any certificate since $f_1,\dots,f_m$ would conform to the same certificate. Hence $g_1,\dots,g_m$ are the claimed polymorphism.

\section{Symmetric binary predicates: triviality}
\label{sec:symmetric}

In this section we determine which non-degenerate symmetric predicates over $\{0,1\}$ are trivial:

\thmsymmetric*

We remind the reader that $\Phi_\id = \{(\id,\dots,\id)\}$ and $\Phi_\neg = \{\id,\neg\}^m$ (where $m$ is the arity of $P$). Here $\id(x) = x$ and $\neg(x) = \bar{x}$.

\subsection*{Necessity}

The easy part of the proof is showing that $P$ is not $\Phi_{\neg}$-trivial or not $\Phi_{\id}$-trivial in the stated cases. In the cases listed for $\Phi_{\neg}$, we exhibit a $2$-ary polymorphism in which all functions depend on both coordinates:

\begin{itemize}
    \item $P$ consists of all vectors of even parity.

    The polymorphism is $f_1(x) = \cdots = f_m(x) = x_1 \oplus x_2$.

    \item $P$ consists of all vectors of odd parity.

    The polymorphism is $f_1(x) = \cdots = f_{m-1}(x) = x_1 \oplus x_2$ and $f_m(x) = \overline{x_1 \oplus x_2}$.

    We could also take the $3$-ary polymorphism $f_1(x) = \cdots = f_m(x) = x_1 \oplus x_2 \oplus x_3$, which has the advantage that all functions are the same.

    \item $P$ consists of all vectors of weight at least $w$, possibly with the addition of the all-zero vector.

    The polymorphism is $f_1(x) = \cdots = f_m(x) = x_1 \lor x_2$.

    \item $P$ consists of all vectors of weight at most $w$, possibly with the addition of the all-one vector.

    The polymorphism is $f_1(x) = \cdots = f_m(x) = x_1 \land x_2$.
\end{itemize}

Finally, if $P$ is closed under complementation, then it is not $\Phi_\id$-trivial due to the $1$-ary polymorphism $f_1(x) = \cdots = f_m(x) = \bar{x}$.

\subsection*{Sufficiency}

Let $\Phi \in \{\Phi_\id, \Phi_{\neg}\}$, and suppose that $P$ is not $\Phi$-trivial. Our goal is to show that one of the cases in the statement of the theorem holds. It will be useful to represent $P$ by the set $W$ of Hamming weights of vectors in $P$.

Applying \Cref{thm:main,thm:reduction}, one of the following cases must hold:

\begin{enumerate}
    \item There exists a $1$-ary polymorphism $f_1,\dots,f_m$ where $f_i(x) \in \{0,1,x\}$, other than $f_1 = \cdots = f_m = x$, which is not of certificate type.

    \item When $\Phi = \Phi_\id$: There exists a $1$-ary polymorphism $f_1,\dots,f_m$ where $f_i(x) \in \{x,\bar{x}\}$, other than $f_1 = \cdots = f_m = x$.

    \item $P$ is closed under setting $i$ to $b$, for some $i,b$.

    \item $P$ has an AND/OR polymorphism.

    \item $P$ has a Latin square polymorphism.
\end{enumerate}

We consider each of these cases below.
The argument for the second case also proves the ``furthermore'' clause of the theorem.

\paragraph{$P$ is closed under setting some coordinates to constants, and the constant coordinates do not constitute a certificate.}

Suppose that for all $y \in P$, if we set the first $a_0$ coordinates to $0$ and the last $a_1$ coordinates to $1$ then the resulting vector is also in $P$, where $a_0 + a_1 > 0$. We denote this operation by $\mathcal{O}(y)$.
Since the constant coordinates do not constitute a certificate, there is some $y_0 \notin P$ such that $\mathcal{O}(y_0) \notin P$.

For any $y \in \{0,1\}^m$, the weight of $\mathcal{O}(y)$ is always  in the range $\{a_1, \dots, m-a_0\}$. Therefore $W \cap \{a_1, \dots, m-a_0\} \neq \emptyset$. Considering the vector $\mathcal{O}(y_0)$, also $W \cap \{a_1, \dots, m-a_0\} \neq \{a_1, \dots, m-a_0\}$. These properties will lead to a contradiction if $a_0, a_1 > 0$, and will allow us to uncover the structure of $W$ if $a_0 = 0$ or $a_1 = 0$.

The main observation is the following lemma.

\begin{lemma} \label{lem:symmetric-const}
Let $w \in W$ be such that $w = a_1 + b_1$ and $m-w = a_0 + b_0$, where $b_0,b_1 \geq 0$ (equivalently, $w \in \{a_1,\dots,m-a_0\}$).

If $a_0,b_1 > 0$ then $w - 1 \in W$. If $a_1,b_0 > 0$ then $w + 1 \in W$.
\end{lemma}
\begin{proof}
If $a_0,b_1 > 0$ then consider $y = 0^{a_0-1} 1 \; 0^{b_0+1} 1^{b_1-1} \; 1^{a_1} \in P$.
We have $\mathcal{O}(y) = 0^{a_0} \; 0^{b_0+1} 1^{b_1-1} \; 1^{a_1} \in P$, and so $w - 1 \in W$.

The argument in the case $a_1,b_0 > 0$ is completely analogous.
\end{proof}

Suppose first that $a_0,a_1 > 0$. The lemma implies that if $w-1,w \in \{a_1,\dots,m-a_0\}$ and $w \in W$ then $w-1 \in W$; and that if $w,w+1 \in \{a_1,\dots,m-a_0\}$ and $w \in W$ then $w+1 \in W$. It follows that $W \cap \{a_1,\dots,m-a_0\}$ is either empty or contains all of $\{a_1,\dots,m-a_0\}$, and so we reach a contradiction.

\smallskip

Suppose next that $a_1 = 0$. In this case the lemma still implies that if $w \in \{1,\dots,m-a_0\}$ and $w \in W$ then $w-1 \in W$. Therefore $W \cap \{0,\dots,m-a_0\} = \{0,\dots,\hat w\}$, where $0 \leq \hat w < m-a_0$. In particular, $m-a_0 \notin W$.

We claim that $W = \{0,\dots,\hat w\}$, which is one of the cases in the statement of the theorem (note that $\hat w \neq 0$ due to non-degeneracy). Indeed, suppose that $w \in W$ for some $w > m-a_0$. Then $y = 0^{m-w} 1^{a_0-(m-w)} \; 1^{m-a_0} \in P$, and so $\mathcal{O}(y) = 0^{a_0} \; 1^{m-a_0} \in P$. However, this implies that $m-a_0 \in W$, and we reach a contradiction.

\smallskip

Similarly, if $a_0 = 0$ then $W = \{\hat w,\dots,m\}$ for some $a_1 < \hat w < m$.

\paragraph{When $\Phi = \Phi_\id$: $P$ is closed under XORing with some $v \neq 0$.}

If $v = 1^m$ then $P$ is closed under complementation, which is one of the cases in the statement of the theorem. Therefore we can assume that $0 < |v| < m$ (where $|v|$ is the Hamming weight of $v$).

By symmetry, $P$ is closed under XORing with any vector of Hamming weight $|v|$. In particular, it is invariant under XORing with both $01 1^{|v|-1} 0^{m-|v|-1}$ and $101^{|v|-1} 0^{m-|v|-1}$, and so under XORing with their XOR, which is $11 0^{m-2}$.
This implies that if $w \in W$ satisfies $w \geq 2$ then $w-2 \in W$: we get this by considering $1^w 0^{m-w} \in P$. Similarly, if $w \in W$ satisfies $w \leq m-2$ then $w+2 \in W$, considering $0^{m-w} 1^w \in P$.

Thus $W$ either contains all odd numbers in $\{0,\dots,m\}$ or none of them, and similarly it either contains all even numbers in $\{0,\dots,m\}$ or none of them. Since $P$ is non-degenerate, this implies that $P$ consists of either all vectors of even parity or of all vectors of odd parity, both of which are cases in the statement of the theorem.

\smallskip

This argument shows that if $P$ is closed under XORing with $v \neq 0$ then either $v = 1^m$ or $P$ consists of all vectors of even parity or of all vectors of odd parity, and consequently it is not $\Phi_{\neg}$-trivial. This proves the ``furthermore'' clause of the theorem.

\paragraph{$P$ is closed under setting a single coordinate to a constant.}

Suppose first that $P$ is closed under setting a single coordinate to $0$. This implies that if $w \in W$ is positive then also $w-1 \in W$, and so $W = \{0, \dots, \hat{w}\}$ for some $\hat{w}$.

Similarly, if $P$ is closed under setting a single coordinate to $1$ then $W = \{\hat{w}, \dots, m\}$ for some $\hat{w}$.

\paragraph{$P$ has an AND/OR polymorphism.}

Recall that an AND/OR polymorphism is a $2$-ary polymorphism $f_1,\dots,f_m$ were for every $i$, either $f_i(x) = x_1 \land x_2$ or $f_i(x) = x_1 \lor x_2$.

\smallskip

Suppose first that there are two different coordinates $i_1,i_2$ such that $f_{i_1}(x) = f_{i_2}(x) = x_1 \land x_2$. If $w \in W$ is such that $0 < w < m$ then we can find a vector $y \in P$ of weight $w$ such that $(y_{i_1},y_{i_2}) = (0,1)$. Let $z \in P$ be obtained from $w$ by switching coordinates $i_1$ and $i_2$. Then $(f_1(y_1,z_1),\dots,f_m(y_m,z_m)) = y|_{i_2 \gets 0}$ has weight $w-1$, showing that $w-1 \in P$. Thus either $W = \{0,\dots,\hat{w}\}$ or $W = \{0,\dots,\hat{w}\} \cup \{m\}$, for some $\hat{w}$, both of which are cases in the statement of the theorem.

\smallskip

Similarly, if there are two different coordinates $i_1,i_2$ such that $f_{i_1}(x) = f_{i_2}(x) = x_1 \lor x_2$ then either $W = \{\hat{w},\dots,m\}$ or $W = \{0\} \cup \{\hat{w},\dots,m\}$ for some $\hat{w}$.

\smallskip

If none of these cases happens then $m \leq 2$. There are no non-degenerate predicates for $m = 1$. When $m = 2$, the non-degenerate predicates are:
\begin{itemize}
\item $W = \{0,1\}$: weight at most $1$.
\item $W = \{0,2\}$: weight at most $0$ together with the all-one vector, or weight at least $2$ together with the all-one vector.
\item $W = \{1,2\}$: weight at least $1$.
\end{itemize}
Each of these is one of the cases in the statement of the theorem.

\paragraph{$P$ has a Latin square polymorphism.}

Recall that a Latin square polymorphism is a $2$-ary polymorphism $f_1,\dots,f_m$ such that if we view $f_i$ as a $2 \times 2$ square then it is a Latin square. Thus $f_i(x) = x_1 \oplus x_2 \oplus v_i$ for some vector $v \in \{0,1\}^m$.

In this case $P$ is closed under the operation $(y,z) \mapsto y \oplus z \oplus v$. If $W = \{0,m\}$ then we are done since this is one of the cases in the statement of the theorem. Otherwise, $P$ contains some $z_0 \neq v, \bar{v}$. Thus $P$ is closed under XORing with $z_0 \oplus v \neq 0^m, 1^m$, and so as shown above, $P$ consists of all vectors of some fixed parity.

\section{Symmetric binary predicates: classification}
\label{sec:symmetric-classification}

In this section we determine all polymorphisms for all non-degenerate symmetric predicates over $\{0, 1\}$:

\thmsymmetricclassification*

\Cref{thm:symmetric} with $\Phi = \Phi_\neg$ immediately implies \Cref{itm:sc-trivial}. \Cref{itm:sc-odd,itm:sc-equal} are trivial. We prove the remaining items one by one.

\begin{proof}[Proof of \Cref{itm:sc-parity}]
This item is known in theoretical computer science, and can be easily proved using Fourier analysis. For the sake of completeness, we provide a combinatorial proof.

We start by verifying that polymorphisms of the stated form are indeed polymorphisms. Suppose that $x^{(1)},\dots,x^{(m)} \in \{0,1\}^n$ satisfy $x^{(1)}_j \oplus \cdots \oplus x^{(n)}_j = b$ for all $j$. Then
\[
 f_1(x^{(1)}) \oplus \cdots \oplus f_m(x^{(m)}) =
 b_1 \oplus \cdots \oplus b_m \oplus \bigoplus_{j \in J} \bigl(x^{(1)}_j \oplus \cdots \oplus x^{(m)}_j\bigr) = b_1 \oplus \cdots \oplus b_m \oplus |J| b = b. 
\]

In the other direction, let us say that a coordinate $j \in [n]$ is \emph{sensitive} if there exist a coordinate $i \in [m]$ and an input $x$ such that $f_i(x \oplus j) = \overline{f_i(x)}$, where $x \oplus j$ results from $x$ by flipping the $j$'th coordinate.

We claim that if $j$ is sensitive then in fact $f_i(x \oplus j) = \overline{f_i(x)}$ for all $i,x$. Denoting by $J$ the set of sensitive coordinates, this implies the claimed structure up to the condition on the $b_i$'s. The condition, in turn, follows from the calculation above.

To prove the claim, suppose that $f_{i'}(x' \oplus j) = \overline{f_{i'}(x')}$ for some $i',x'$. 
We first prove the claim for all $i \neq i'$. Since $m \ge 3$, we can construct inputs $x^{(1)},\dots,x^{(m)}$ such that $x^{(i')} = x'$, $x^{(i)} = x$, and $x^{(1)}_k \oplus \cdots \oplus x^{(m)}_k = b$ for all $k$. Since $f_1,\dots,f_m$ is a polymorphism, this implies that $f_1(x^{(1)}) \oplus \cdots \oplus f_m(x^{(m)}) = b$.
If we flip the $j$'th coordinate of $x^{(i)},x^{(i')}$ then the new input $y^{(1)},\dots,y^{(m)}$ still satisfies $y^{(1)}_j \oplus \cdots \oplus y^{(m)}_j = b$, and so $f_1(y^{(1)}) \oplus \cdots \oplus f_m(y^{(m)}) = b$. This implies that
\[
 f_i(x^{(i)}) \oplus f_{i'}(x^{(i')}) = f_i(y^{(i)}) \oplus f_{i'}(y^{(i')}) = f_i(x^{(i)} \oplus j) \oplus \overline{f_{i'}(x^{(i')})}.
\]
It follows that $f_i(x^{(i)} \oplus j) = \overline{f_i(x^{(i)})}$, as claimed.
The claim for $i = i'$ now follows by the same argument using a different $i'$.
\end{proof}

\begin{proof}[Proof of \Cref{itm:sc-atmost}]
The second claim follows from the first, so we only prove the first one.

Suppose first that $F_1,\dots,F_m$ are $(w+1)$-wise intersecting, and let $S_1,\dots,S_m$ be such that for all $j$, at most $w$ of the sets $S_1,\dots,S_m$ contain $j$. We need to show that $S_i \in F_i$ for at most $w$ many $i$'s. This holds since otherwise the families are not $(w+1)$-wise intersecting.

The other direction is similar. Suppose that $F_1,\dots,F_m$ are such that whenever $S_1,\dots,S_m$ are as before, then $S_i \in F_i$ for at most $w$ many $1$'s. We need to show that $F_1,\dots,F_m$ are $(w+1)$-wise intersecting. If not, suppose without loss of generality that $S_1 \in F_1, \dots, S_{w+1} \in F_{w+1}$ have empty intersection. Taking $S_{w+2} = \cdots = S_m = \emptyset$, we obtain a contradiction.
\end{proof}

\begin{proof}[Proof of \Cref{itm:sc-atmost-m}]
The second claim follows from the first, so we only prove the first one.

We start with the `if' direction. If $f_1 = \cdots = f_m$ is an AND then $f_1,\dots,f_m$ is a polymorphism of $P$ since $(1,\dots,1) \in P$ (this takes care of the case of a degenerate AND, which is the constant~$1$ function) and $P$ is closed under AND. If at least $m-w$ of $f_1,\dots,f_m$ are constant~$0$ then they are of certificate type, and in particular, a polymorphism.

We continue with the `only if' direction. The proof is in two steps. We first show that for each subset $I \subseteq [m]$ of size $w+2$, either all $f_i$ for $i \in I$ are equal to the same AND, or at least two are constant~$0$. The `only if' direction then easily follows.

\smallskip

We start with the first step, considering for simplicity $I = [w+2]$. We first observe that $P|_{[w+2]}$ consists of all vectors whose weight is not exactly~$w+1$. Therefore the families $F_1,\dots,F_{w+2} \subseteq 2^{[m]}$ corresponding to the functions $f_1,\dots,f_{w+2}$ satisfy the following condition: if $S_1,\dots,S_{w+2}$ are such that no $j$ is contained in exactly $w+1$ of the sets, that $S_i \in F_i$ cannot hold for exactly $w+1$ many $i$'s.

We first show that either at least two of the families are empty, or all families are non-empty. Indeed, suppose that $F_1,\dots,F_{w+1}$ are non-empty, and choose $S_1 \in F_1, \dots, S_{w+1} \in F_{w+1}$ arbitrarily. Let $S_{w+2} = S_1 \cap \cdots \cap S_{w+1}$. By construction, no $j$ is contained in exactly $w+1$ of the sets, and so $S_{w+2}$ must belong to $F_{w+2}$, showing that $F_{w+2}$ is non-empty as well. A similar argument rules out any other family being the only empty one.

If at least two of the families are empty then we are done, so assume that all families are non-empty, say witnessed by $S_i \in F_i$. Let $A \in F_{i_0}$ be a set of minimum size among all sets in all families. We will show that all $F_i$ consist of all sets containing $A$, and so $f_1(x) = \cdots = f_{w+2}(x) = \bigwedge_{j \in A} x_j$.

We start by showing that all sets in $F_i$ for $i \neq i_0$ contain $A$. Consider any $i_1 \neq i_0$ and any $B \in F_{i_1}$, choose some $i_2 \neq i_0,i_1$, and let $C = A \cap B \cap \bigcap_{i \neq i_0,i_1,i_2} S_i$. The sets $A,B,C,(S_i)_{i \neq i_0,i_1,i_2}$ satisfy the condition that no $j$ belongs to exactly $w+1$ of them, and so not exactly $w+1$ of them belong to their respective families. By construction, all sets other than possibly $C$ belong to their respective families, hence $C \in F_{i_2}$. This implies that $B \supseteq A$ by the definition of $A$.

Next, we show that all $F_i$ contain $A$. This is clear for $i = i_0$, so suppose that $i \neq i_0$. We repeat the argument above with $i_2 = i$, an arbitrary $i_1 \neq i_0,i$, and $B = S_{i_1}$. The argument shows that $C \in F_i$, where $C \subseteq A$. The definition of $A$ implies that $C = A$.
We can now replace $i_0$ with any other $i$ in the argument above to deduce that all sets in $F_{i_0}$ also contain $A$.

It remains to show that each $F_i$ consists of all sets containing $A$. Indeed, let $B \supseteq A$. We will show that $B \in F_{w+2}$; the same argument works for other $i$. Consider the sets $A,\dots,A,B$. Since $B \supseteq A$, every element $j$ belongs to either $1$ or $w+2$ of the families. Since $A \in F_1,\dots,F_{w+1}$, this implies that $B \in F_{w+2}$, as needed.

\smallskip

We proceed with the second step. We consider two cases. If $f_i$ is the constant~$0$ function for some $i$ then the constant case must hold for all $I$ (since the constant~$0$ function cannot be written as an AND), and so at least $m-w$ of the functions are constant~$0$. Otherwise, the AND case must hold for all $I$, and we are done.
\end{proof}

\section{Relation to impossibility domains}
\label{sec:DH}

In this section we relate $\Phi$-triviality to the notion of impossibility domain.

\thmDH*

One direction of the theorem is almost immediate: if $P$ is $\Phi$-trivial then it is an impossibility domain with respect to unanimity. We show this by proving the contrapositive: if $P$ is a possibility domain with respect to unanimity then it is not $\Phi$-trivial.

Suppose that $P$ is a possibility domain with respect to unanimity. Then there is a unanimous polymorphism $f_1,\dots,f_m$ which is not of the form $f_1(x) = \cdots = f_m(x) = x_j$.

We claim that $f_1,\dots,f_m$ is of neither dictatorial type nor certificate type, and so $P$ is not $\Phi$-trivial. We prove this by considering both cases.

Suppose first that $f_1,\dots,f_m$ are of dictatorial type: there exists $j \in [n]$ such that $f_i(x) = \phi_i(x_j)$. Since $f_i$ is unanimous, we have $\phi_i(b) = f_i(b,\dots,b) = b$, and so $f_i(x) = x_j$. This contradicts the assumption on $f_1,\dots,f_m$.

If $f_1,\dots,f_m$ conform to a certificate $\rho$ then $\dom\rho = \emptyset$, since no $f_i$ is constant (by unanimity). The empty certificate cannot be a certificate since $P$ is non-degenerate.

\medskip

The other direction is an application of our main theorems. We show that if $P$ is not $\Phi$-trivial then it is a possibility domain with respect to unanimity.

Since $P$ is not $\Phi$-trivial, according to \Cref{thm:main} it is not $\Phi$-trivial for $n = 2$.
Since it is $\Phi$-trivial for $n = 1$, one of the three cases in \Cref{thm:reduction} must hold. In each of these cases we will show that $P$ is a possibility domain with respect to unanimity by constructing a unanimous non-dictatorial polymorphism.

\begin{enumerate}
\item There is a coordinate $i_0 \in [m]$ and $\sigma_0 \in \Sigma_{i_0}$ such that $P$ is closed under setting $i_0$ to $\sigma_0$.

We construct unanimous functions $f_i\colon \Sigma_i^2 \to \Sigma_i$ as follows. If $i \neq i_0$ then $f_i(x) = x_1$. The function $f_{i_0}$ is defined as follows: $f_i(x) = x_1$ if $x_2 \neq \sigma_0$, and $f_i(x) = \sigma_0$ if $x_2 = \sigma_0$. Since $f_i$ depends on both arguments, $f_1,\dots,f_m$ is non-dictatorial. It remains to prove that $f_1,\dots,f_m$ is a polymorphism of $P$.

Suppose that $y,z \in P$. If $z_{i_0} \neq \sigma_0$ then $(f_1(y_1,z_1),\dots,f_m(y_m,z_m)) = y \in P$, and if $z_{i_0} = \sigma_0$ then $(f_1(y_1,z_1),\dots,f_m(y_m,z_m)) = y|_{i_0 \gets \sigma_0} \in P$.

\item $P$ has a non-dictatorial AND/OR polymorphism. That is, there are functions $f_i\colon \Sigma_i^2 \to \Sigma_i$ such that (i) if $|\Sigma_i| > 2$ then $f_i(x) = x_1$, (ii) if $|\Sigma_i| = 2$ then $f_i \in \{ x_1 \land x_2, x_1 \lor x_2 \}$, where we assume without loss of generality that $\Sigma_i = \{0, 1\}$ in this case.

Since AND/OR polymorphisms are unanimous, we are immediately done.

\item $P$ has a Latin square polymorphism. That is, there are functions $f_i\colon \Sigma_i^2 \to \Sigma_i$ where each $f_i$ is a Latin square (viewed as a matrix, each row and each column is a permutation).

Let $f_i(x) = \pi_{i,x_1}(x_2)$,  where $\pi_{i,x_1}$ is some permutation of $\Sigma_i$. Then $f_i^{\circ 2}(x_1,x_2,x_3) := f_i(x_1,f_i(x_2,x_3)) = \pi_{i,x_1}(\pi_{i,x_2}(x_3))$. Observe that $f_1^{\circ 2},\dots,f_m^{\circ 2}$ is also a polymorphism of $P$. Indeed, if $y,z,w \in P$ then $(f_1(z_1,w_1),\dots,f_m(z_m,w_m)) \in P$, and so $(f_1(y_1,f_1(z_1,w_1)),\dots,f_m(y_m,f_m(z_m,w_m)) \in P$.

More generally, we can define $f_i^{\circ r}\colon \Sigma_i^{r+1} \to \Sigma_i$ by the recursive formula $f_i^{\circ 1}(x) = f_i(x)$ and $f_i^{\circ r+1}(x) = f_i(x_1,f_i^{\circ r}(x_2,\dots,x_{r+1}))$. For every $r$, $f_1^{\circ r},\dots,f_m^{\circ r}$ is a polymorphism of $P$ which is not dictatorial. In order to complete the proof, we need to find $r$ for which the functions $f_i^{\circ r}$ are unanimous.

Observe that $f_i^{\circ r}(\sigma,\dots,\sigma) = \pi_{i,\sigma}^r(\sigma)$, where the power is taken in the group of symmetries of $\Sigma_i$. Taking $r = \prod_{i=1}^m |\Sigma_i|!$, we have $\pi_{i,\sigma}^r = \id$ and so $f_i$ is unanimous.
\end{enumerate}

\bibliographystyle{alphaurl}
\bibliography{biblio}

\begin{thebibliography}{BBKO21}

\bibitem[AGH17]{AGH17}
Per Austrin, Venkatesan Guruswami, and Johan H{\aa}stad.
\newblock {$(2+\varepsilon)$}-{S}at is {NP}-hard.
\newblock {\em SIAM J. Comput.}, 46(5):1554--1573, 2017.
\newblock \href {https://doi.org/10.1137/15M1006507}
  {\path{doi:10.1137/15M1006507}}.

\bibitem[Arr50]{Arrow50}
Kenneth~J. Arrow.
\newblock A difficulty in the theory of social welfare.
\newblock {\em J. of Political Economy}, 58:328--346, 1950.

\bibitem[Arr63]{Arrow63}
Kenneth~J. Arrow.
\newblock {\em Social choice and individual values}.
\newblock John Wiley and Sons, 1963.

\bibitem[BBKO21]{BBKO21}
Libor Barto, Jakub Bul\'in, Andrei Krokhin, and Jakub Opr{\v s}al.
\newblock Algebraic approach to promise constraint satisfaction.
\newblock {\em J. ACM}, 68(4):Art. 28, 66, 2021.
\newblock \href {https://doi.org/10.1145/3457606} {\path{doi:10.1145/3457606}}.

\bibitem[BG16]{BG16}
Joshua Brakensiek and Venkatesan Guruswami.
\newblock New hardness results for graph and hypergraph colorings.
\newblock In {\em 31st {C}onference on {C}omputational {C}omplexity}, volume~50
  of {\em LIPIcs. Leibniz Int. Proc. Inform.}, pages Art. No. 14, 27. Schloss
  Dagstuhl. Leibniz-Zent. Inform., Wadern, 2016.

\bibitem[BG21]{BG21}
Joshua Brakensiek and Venkatesan Guruswami.
\newblock Promise constraint satisfaction: algebraic structure and a symmetric
  {B}oolean dichotomy.
\newblock {\em SIAM J. Comput.}, 50(6):1663--1700, 2021.
\newblock \href {https://doi.org/10.1137/19M128212X}
  {\path{doi:10.1137/19M128212X}}.

\bibitem[BJ03]{BulatovJeavons03}
Andrei~A. Bulatov and Peter Jeavons.
\newblock An algebraic approach to multi-sorted constraints.
\newblock In {\em International Conference on Principles and Practice of
  Constraint Programming}, 2003.
\newblock URL: \url{https://api.semanticscholar.org/CorpusID:990270}.

\bibitem[BK25]{BartoKapytka25}
Libor Barto and Maryia Kapytka.
\newblock Multisorted {B}oolean clones determined by binary relations up to
  minion homomorphisms.
\newblock {\em Algebra Universalis}, 86(1):Paper No. 1, 36, 2025.
\newblock \href {https://doi.org/10.1007/s00012-024-00878-0}
  {\path{doi:10.1007/s00012-024-00878-0}}.

\bibitem[DH09]{DokowHolzman09}
Elad Dokow and Ron Holzman.
\newblock Aggregation of binary evaluations for truth-functional agendas.
\newblock {\em Soc. Choice Welf.}, 32(2):221--241, 2009.
\newblock \href {https://doi.org/10.1007/s00355-008-0320-1}
  {\path{doi:10.1007/s00355-008-0320-1}}.

\bibitem[DH10a]{DokowHolzman10a}
Elad Dokow and Ron Holzman.
\newblock Aggregation of binary evaluations.
\newblock {\em J. Econom. Theory}, 145(2):495--511, 2010.
\newblock \href {https://doi.org/10.1016/j.jet.2007.10.004}
  {\path{doi:10.1016/j.jet.2007.10.004}}.

\bibitem[DH10b]{DokowHolzman10b}
Elad Dokow and Ron Holzman.
\newblock Aggregation of binary evaluations with abstentions.
\newblock {\em J. Econom. Theory}, 145(2):544--561, 2010.
\newblock \href {https://doi.org/10.1016/j.jet.2009.10.015}
  {\path{doi:10.1016/j.jet.2009.10.015}}.

\bibitem[DH10c]{DokowHolzman10c}
Elad Dokow and Ron Holzman.
\newblock Aggregation of non-binary evaluations.
\newblock {\em Adv. in Appl. Math.}, 45(4):487--504, 2010.
\newblock \href {https://doi.org/10.1016/j.aam.2010.02.005}
  {\path{doi:10.1016/j.aam.2010.02.005}}.

\bibitem[DL07]{DietrichList07}
Franz Dietrich and Christian List.
\newblock Arrow's theorem in judgment aggregation.
\newblock {\em Soc. Choice Welf.}, 29(1):19--33, 2007.
\newblock \href {https://doi.org/10.1007/s00355-006-0196-x}
  {\path{doi:10.1007/s00355-006-0196-x}}.

\bibitem[DL13]{DietrichList13}
Franz Dietrich and Christian List.
\newblock Propositionwise judgment aggregation: the general case.
\newblock {\em Soc. Choice Welf.}, 40(4):1067--1095, 2013.
\newblock \href {https://doi.org/10.1007/s00355-012-0661-7}
  {\path{doi:10.1007/s00355-012-0661-7}}.

\bibitem[FR86]{FishburnRubinstein86}
Peter~C. Fishburn and Ariel Rubinstein.
\newblock Aggregation of equivalence relations.
\newblock {\em J. Classification}, 3:61--65, 1986.

\bibitem[Gei68]{Geiger68}
David Geiger.
\newblock Closed systems of functions and predicates.
\newblock {\em Pacific J. Math.}, 27:95--100, 1968.
\newblock URL: \url{http://projecteuclid.org/euclid.pjm/1102985564}.

\bibitem[Gib14]{Gibbard14}
Allan~F. Gibbard.
\newblock Intransitive social indifference and the {A}rrow dilemma.
\newblock {\em Review of Economic Design}, 18:3--10, 2014.

\bibitem[JM59]{JanovMucnik59}
Ju.~I. Janov and A.~A. Mu{\v c}nik.
\newblock Existence of {$k$}-valued closed classes without a finite basis.
\newblock {\em Dokl. Akad. Nauk SSSR}, 127:44--46, 1959.

\bibitem[LP02]{ListPettit02}
Christian List and Philip Pettit.
\newblock Aggregating sets of judgments: An impossibility result.
\newblock {\em Econ. Philos.}, 18:89--110, 2002.

\bibitem[LPW18]{LPW18}
Erkko Lehtonen, Reinhard P\"oschel, and Tam\'as Waldhauser.
\newblock Reflections on and of minor-closed classes of multisorted operations.
\newblock {\em Algebra Universalis}, 79(3):Paper No. 71, 19, 2018.
\newblock \href {https://doi.org/10.1007/s00012-018-0549-1}
  {\path{doi:10.1007/s00012-018-0549-1}}.

\bibitem[Mos09]{Mossel09}
Elchanan Mossel.
\newblock Arrow's impossibility theorem without unanimity, 2009.
\newblock URL: \url{https://arxiv.org/abs/0901.4727}, \href
  {https://arxiv.org/abs/0901.4727} {\path{arXiv:0901.4727}}.

\bibitem[Mos12]{Mossel12}
Elchanan Mossel.
\newblock A quantitative {A}rrow theorem.
\newblock {\em Probab. Theory Related Fields}, 154(1-2):49--88, 2012.
\newblock \href {https://doi.org/10.1007/s00440-011-0362-7}
  {\path{doi:10.1007/s00440-011-0362-7}}.

\bibitem[NP02]{NehringPuppe02}
Klaus Nehring and Clemens Puppe.
\newblock Strategy-proof social choice on single-peaked domains: possibility,
  impossibility and the space between.
\newblock Mimeo, 2002.

\bibitem[Pip02]{Pippenger02}
Nicholas Pippenger.
\newblock Galois theory for minors of finite functions.
\newblock {\em Discrete Math.}, 254(1-3):405--419, 2002.
\newblock \href {https://doi.org/10.1016/S0012-365X(01)00297-7}
  {\path{doi:10.1016/S0012-365X(01)00297-7}}.

\bibitem[Pos20]{Post20}
Emil~L. Post.
\newblock Determination of all closed systems of truth tables.
\newblock {\em Bull. Amer. Math. Soc.}, 26:437, 1920.

\bibitem[Pos42]{Post42}
Emil~L. Post.
\newblock {\em The Two-Valued Iterative Systems of Mathematical Logic}.
\newblock Princeton University Press, Princeton, 1942.
\newblock URL: \url{https://doi.org/10.1515/9781400882366} [cited 2025-02-07],
  \href {https://doi.org/doi:10.1515/9781400882366}
  {\path{doi:doi:10.1515/9781400882366}}.

\bibitem[RF86]{RubinsteinFishburn86}
Ariel Rubinstein and Peter~C. Fishburn.
\newblock Algebraic aggregation theory.
\newblock {\em J. Economic Theory}, 38(1):63--77, February 1986.

\bibitem[SX15]{SzegedyXu15}
Mario Szegedy and Yixin Xu.
\newblock Impossibility theorems and the universal algebraic toolkit.
\newblock Papers 1506.01315, arXiv.org, June 2015.
\newblock URL: \url{https://ideas.repec.org/p/arx/papers/1506.01315.html}.

\bibitem[Wil72]{Wilson72}
Robert Wilson.
\newblock Social choice theory without the {P}areto principle.
\newblock {\em J. of Economic Theory}, 5:478--486, 1972.

\bibitem[Yab54]{Yablonski54}
S.~V. Yablonski{\u \i}.
\newblock On functional completeness in a three-valued calculus.
\newblock {\em Doklady Akad. Nauk SSSR (N.S.)}, 95:1153--1155, 1954.

\bibitem[Zhu15]{Zhuk15}
Dmitriy Zhuk.
\newblock The lattice of all clones of self-dual functions in three-valued
  logic.
\newblock {\em J. Mult.-Valued Logic Soft Comput.}, 24(1-4):251--316, 2015.

\end{thebibliography}

\if01
\section*{Literature summary}

\begin{itemize}
    \item Arrow~\cite{Arrow50} required monotonicity. Arrow~\cite{Arrow63} only required unanimity (``unanimity'' is also known as ``Pareto condition'', ``idempotence'', ``constancy'', ``faithfulness'', ``systematicity''). Wilson~\cite{Wilson72} considered some weakening, and Mossel~\cite{Mossel09} (later incorporated to Mossel~\cite{Mossel12}) proved the general case (without any assumptions).
    \item Dokow and Holzman~\cite{DokowHolzman10a} consider the binary case (alphabet $\{0, 1\}$) and characterized the predicates whose only unanimous polymorphisms are dictators (the easy direction [necessity] also appears in Dietrich and List~\cite{DietrichList07}). This follows up on an earlier result of Nehring and Puppe~\cite{NehringPuppe02} which also assumed monotonicity. Nehring and Puppe showed that triviality is equivalent to \emph{total blockedness}. In the more general setting of Dokow and Holzman, we also need to assume that the predicate is not an affine subspace. Similar characterizations under various definitions of ``impossibility domain'' are in Dietrich and List~\cite{DietrichList13}.

    Dokow and Holzman~\cite{DokowHolzman09} considered predicates of the form $(x,f(x))$, and determined all unanimous polymorphisms. They also showed that if we replace unanimity by nonconstant then the only new polymorphisms we get are anti-dictators.

    Dokow and Holzman~\cite{DokowHolzman10c} considered predicates over larger alphabets, under either unanimity or supportiveness. They gave sufficient conditions for dictators to be the only solutions (for supportiveness), and proved their necessity for alphabets of size~$3$. Their work was completed by Szegedy and Xu~\cite{SzegedyXu15} using completely different techniques, who gave a reduction to the case $n = 3$ (for supportiveness) and $n = |\Sigma|$ (for idempotence). In the binary case, they also proved that if a domain supports non-dictatorial polymorphisms then it supports polymorphisms in which each function is one of AND, OR, 3MAJ, 3XOR.

    Dokow and Holzman~\cite{DokowHolzman10b} considered binary evaluations with abstentions, meaning that the inputs are over $0,1,*$. Every input needs to be consistent with a full assignment. The output is either also $0,1,*$ or must be $0,1$ (in this case the input and output predicates are different!). Apart from dictators, we also allow oligarchies. A related work is that of Gibbard~\cite{Gibbard14}.

    \item Rubinstein and Fishburn~\cite{RubinsteinFishburn86,FishburnRubinstein86} consider predicates over arbitrary alphabets, as well as the problem of determining for which ones are all unanimous polymorphisms dictatorial. They also have some initial results, such as classifying polymorphisms of affine subspaces. The second paper considers a particular predicate (equivalence relations), determining all polymorphisms (conjunctions).

    A similar but more restricted setting was considered later by List and Pettit~\cite{ListPettit02}. [This is less related]
\end{itemize}
\fi

\end{document}